\documentclass[11pt,english,a4paper]{article}
\usepackage[centertags]{amsmath}
\usepackage{amssymb}
\usepackage{amsthm}
\usepackage{makeidx}

\usepackage[margin=1.5in]{geometry}
\usepackage{float}
\usepackage{lscape}
\usepackage{graphicx}
\usepackage{amsmath}
\usepackage[table]{xcolor}

\newtheorem{theorem}{Theorem}[section]
\newtheorem{corollary}[theorem]{Corollary}
\newtheorem{lemma}[theorem]{Lemma}
\newtheorem{proposition}[theorem]{Proposition}
\newtheorem{definition}[theorem]{Definition}
\newtheorem{remark}[theorem]{Remark}

\numberwithin{equation}{section} 

\begin{document}

\title{ On differential relations of $2$-orthogonal polynomials }

\author{T. A. Mesquita\footnote{Corresponding author (taugusta.mesquita@gmail.com)} }
\date{}
\maketitle

\begin{center}
{\scriptsize
Escola Superior de Tecnologia e Gest\~ao, Instituto Polit{\'e}cnico de Viana do Castelo, Rua Escola Industrial e Comercial de Nun' \'Alvares, 4900-347, Viana do Castelo, Portugal, \&\\
Centro de Matem\'{a}tica da Universidade do Porto, Rua do Campo Alegre, 687, 4169-007 Porto, Portugal\\
}
\end{center}

\begin{abstract}

A generic differential operator on the vectorial space of polynomial functions was presented in \cite{PM-TAM-2019}  and applied in the study of differential relations fulfilled by polynomial sequences either orthogonal or $2$-orthogonal.

Using the techniques therein developed, we prove an identity fulfilled by different differential operators and apply it in a systematic approach to the problem of finding polynomial eigenfunctions, assuming that those polynomials constitute a $2$-orthogonal polynomial sequence.  

In particular, we analyse  a third order differential operator that does not increase the degree of polynomials.
\end{abstract}


\section{Notation and basic concepts}

Let $\mathcal{P}$ be the vector space of polynomials with coefficients in $\mathbb{C}$ and let
$\mathcal{P}^{\prime}$ be its topological dual space. We denote by $\langle u ,p\rangle$ the action of the form or linear functional $u \in\mathcal{P}^{\prime }$ on $p\in\mathcal{P}$. In particular,
$\langle u, x^{n}\rangle:=\nolinebreak\left( u \right) _{n},n\geq 0$ represent the moments of $u$.
In the following, we will call polynomial sequence (PS) to any sequence
${\{P_{n}\}}_{n \geq 0}$ such that $\deg P_{n}= n,\; n \geq 0$, that is, for all non-negative integer.
We will also call monic polynomial sequence (MPS) to a PS so that all polynomials have leading coefficient equal to one.

If ${\{P_{n}\}}_{n \geq 0}$ is a MPS, there exists a unique sequence
$\{u_n\}_{n\geq 0}$, $u_n\in\mathcal{P}^{\prime}$, called the dual sequence of $\{P_{n}\}_{n\geq 0}$, such that,
\begin{equation}\label{SucDual}
<u_{n},P_{m}>=\delta_{n,m}\ , \ n,m\ge 0.
\end{equation}

%
\noindent On the other hand, given a MPS ${\{P_{n}\}}_{n \geq 0}$, the expansion of $xP_{n+1}(x)$, defines sequences in $\mathbb{C}$, ${\{\beta_{n}\}}_{n \geq 0}$ and
$\{\chi_{n,\nu}\}_{0 \leq \nu \leq n,\; n \geq 0},$ such that
\begin{eqnarray}
\label{divisao_ci} &&P_{0}(x)=1, \;\; P_{1}(x)=x-\beta_{0},\\
\label{divisao} && xP_{n+1}(x)= P_{n+2}(x)+ \beta_{n+1}P_{n+1}(x)+\sum_{\nu=0}^{n}\chi_{n,\nu}P_{\nu}(x).
\end{eqnarray}
This relation is usually called the structure relation of  ${\{P_{n}\}}_{n \geq 0}$, and ${\{\beta_{n}\}}_{n \geq 0}$ and
$\{\chi_{n,\nu}\}_{0 \leq \nu \leq n,\; n \geq 0}$ are called the structure coefficients (SCs) \cite{theoriealgebrique}. Another useful presentation is the folllowing.
\begin{align*}
&P_{n+2}(x) = (x-\beta_{n+1})P_{n+1}(x)+\sum_{\nu=0}^{n}\chi_{n,\nu}P_{\nu}(x)\, ,\\
&P_{0}(x)=1, \;\; P_{1}(x)=x-\beta_{0}\,.
\end{align*}

When the structure coefficients fulfil $\chi_{n,\nu}=0\; ,\; 0 \leq \nu \leq n-1$ , $\chi_{n,n} \neq 0$, identities \eqref{divisao_ci}-\eqref{divisao} refer to the well known three-term recurrence associated to an orthogonal MPS. More generally, identity \eqref{divisao} may furnish a recurrence relation for a higher order
corresponding to the following notion of orthogonality with respect to $d$ given functionals.

\medskip
\begin{definition}\cite{Douak2classiques,Maroni-Toulouse,Van-Iseg-1987}
Given $\Gamma^{1}, \Gamma^{2},\ldots, \Gamma^{d} \in \mathcal{P}^{\prime}$, $d \geq 1$, the polynomial sequence ${\{P_{n}\}}_{n \geq 0}$ is called d-orthogonal polynomial sequence (d-OPS) with respect to  $\Gamma=(\Gamma^{1},\ldots, \Gamma^{d} )$ if it fulfils 
\begin{equation}\label{d ortogonal}
 \langle \Gamma^{\alpha}, P_{m}P_{n}  \rangle=0, \;\; n \geq md +\alpha,\,\; m \geq 0,
\end{equation}
 \begin{equation}\label{d ortogonal regular}
 \langle \Gamma^{\alpha}, P_{m}P_{md+\alpha-1}  \rangle \neq 0, \;\; m \geq 0,
\end{equation}
for each integer $\alpha=1,\ldots, d$.

%
\begin{lemma}\cite{variations}\label{lema1}
For each $u \in \mathcal{P}^{\prime}$ and each $m \geq 1$, the two following propositions are equivalent.
\begin{description}
\item[a)]  $\langle u, P_{m-1} \rangle \neq 0,\:\: \langle u, P_{n} \rangle=0,\: n \geq m$.
\item[b)] $\exists \lambda_{\nu} \in \mathbb{C},\:\: 0 \leq \nu \leq m-1,\:\: \lambda_{m-1}\neq 0$ such that
$u=\sum_{\nu=0}^{m-1}\lambda_{\nu} u_{\nu}$. \end{description}
\end{lemma}

The conditions (\ref{d ortogonal}) are called the d-orthogonality conditions and the conditions (\ref{d ortogonal regular}) are called the regularity conditions. In this case, the functional $\Gamma$, of dimension $d$, is said regular.
\end{definition}
The  $d$-dimensional functional $\Gamma$ is not unique. Nevertheless, from Lemma \ref{lema1}, we have:
$$\Gamma^{\alpha}=\sum_{\nu=0}^{\alpha-1}\lambda_{\nu}^{\alpha} u_{\nu}, \;\; \lambda_{\alpha-1}^{\alpha} \neq 0, \; 1\leq \alpha \leq d.$$
Therefore, since $U=(u_{0},\ldots, u_{d-1} )$ is unique, we use to consider the canonical functional of dimension $d$, $U=(u_{0},\ldots, u_{d-1} )$, saying that
${\{P_{n}\}}_{n \geq 0}$ is d-orthogonal (for any positive integer $d$) with respect to
$U=(u_{0},\ldots, u_{d-1})$ if
$$ \langle u_{\nu}, P_{m}P_{n}  \rangle=0, \;\; n \geq md +\nu+1,\,\; m \geq 0,$$
$$  \langle u_{\nu}, P_{m}P_{md+\nu}  \rangle \neq 0, \;\; m \geq 0,$$
for each integer $\nu=0,1,\ldots, d-1$.

\begin{theorem}\cite{Maroni-Toulouse} \label{recurrence relation for d ortho}
Let ${\{P_{n}\}}_{n \geq 0}$ be a MPS. The following assertions are equivalent:
\begin{description}
\item[a)]  ${\{P_{n}\}}_{n \geq 0}$ is $d$-orthogonal with respect to $U=(u_{0},\ldots, u_{d-1})$.
\item[b)] ${\{P_{n}\}}_{n \geq 0}$ satisfies a $(d+1)$-order recurrence relation ($d \geq 1$):
$$P_{m+d+1}(x)=(x-\beta_{m+d})P_{m+d}(x)-\sum_{\nu=0}^{d-1}\gamma_{m+d-\nu}^{d-1-\nu}P_{m+d-1-\nu}(x), \:\: m \geq 0,$$
with initial conditions
$$P_{0}(x)=1,\:\: P_{1}(x)=x-\beta_{0}\;\: \textrm{and if }d \geq 2:$$
$$P_{n}(x)=(x-\beta_{n-1})P_{n-1}(x)-\sum_{\nu=0}^{n-2}\gamma_{n-1-\nu}^{d-1-\nu}P_{n-2-\nu}(x), \:\: 2 \leq n \leq  d,$$
and regularity conditions: $\gamma_{m+1}^{0} \neq 0,\: \; m \geq 0.$
\end{description}
\end{theorem}

In this paper, we will focus on $2$-orthogonal MPSs, thus fulfilling the recurrence relation
%
\begin{align}
\nonumber &P_{n+3}(x)=(x-\beta_{n+2})P_{n+2}(x)-\gamma_{n+2}^{1}P_{n+1}(x)-\gamma_{n+1}^{0}P_{n}(x),\\
\nonumber & P_{0}(x)=1,\:\: P_{1}(x)=x-\beta_{0},\:\:
P_{2}(x)=(x-\beta_{1})P_{1}(x)-\gamma_{1}^{1}\, ,\, n \geq 0.
\end{align}
While working solely with $2$-orthogonality it is usual to rename the gammas as follows (cf. \cite{Douak-two-Laguerre})
\begin{align}
\label{RR2ortho} & P_{n+3}(x)=(x-\beta_{n+2})P_{n+2}(x)-\alpha_{n+2}P_{n+1}(x)-\gamma_{n+1}P_{n}(x),\\
\label{RR2ortho-ci} & P_{0}(x)=1,\:\: P_{1}(x)=x-\beta_{0},\:\:
P_{2}(x)=(x-\beta_{1})P_{1}(x)-\alpha_{1}\, ,\; \:n \geq 0.
\end{align}


\section{Differential operators on $\mathcal{P}$ and technical identities}\label{operator}

In this section, we list the main results indicated in \cite{PM-TAM-2019} that will be applied along the text. Later on, we also prove new identities that are the fundamental utensils for the strategy pursued. 
\medskip

Given a sequence of polynomials $\{a_{\nu}(x)\}_{ \nu \geq 0}$, let us consider the following linear mapping $J: \mathcal{P} \rightarrow \mathcal{P}$ (cf. \cite{PM-Korean}, \cite{Pincherle}).
\begin{equation} \label{operatorJ}
J= \sum_{\nu \geq 0} \frac{a_{\nu}(x)}{\nu!} D^{\nu},\quad \deg a_{\nu} \leq \nu,\quad \nu \geq 0.
\end{equation}
Expanding $a_{\nu}(x)$ as follows:
$$a_{\nu}(x)=\sum_{i=0}^{\nu} a_{i}^{[\nu]}x^{i}, $$
and recalling that $ D^{\nu}\left( \xi ^n \right) (x)= \frac{n!}{(n-\nu)!} x^{n-\nu} $, we get the next identities about $J$:
\begin{eqnarray}
\label{Jx^n-short}&& J\left( \xi^n \right) (x) = \sum_{\nu= 0}^{n} a_{\nu} (x) 
\binom{n}{\nu} x^{n-\nu}\, , \\
\label{Jx^n} && J\left( \xi^n \right) (x) = \sum_{\tau= 0}^{n} \left( \sum_{\nu=0}^{\tau}\binom{n}{n-\nu} a_{\tau-\nu}^{[n-\nu]} \right) x^{\tau},\quad n \geq 0.
\end{eqnarray}
Most in particular, a linear mapping $J$ is an isomorphism if and only if 
\begin{equation}\label{iso-conditions}
\deg \left( J\left( \xi^n  \right)(x) \right)= n\;, \;\; n \geq 0, \; \; \textrm{and}\;\; J\left( 1  \right)(x) \neq 0.
\end{equation}
The next result establishes that any operator that does not increase the degree admits an expansion as  \eqref{operatorJ} for certain polynomial coefficients.
\begin{lemma}\label{lemmaPascal} \cite{PM-TAM-2019}
For any linear mapping $J$, not increasing the degree, there exists a unique sequence of polynomials $\{a_{n}\}_{n \geq 0}$, with $\deg a_{n} \leq n$, so that $J$ is read as in \eqref{operatorJ}. Further, the linear mapping $J$ is an isomorphism of $\mathcal{P}$ if and only if
\begin{equation}\label{lambdan}
 \sum_{\mu=0}^{n}\binom{n}{\mu} a_{\mu}^{[\mu]} \neq 0 , \quad n \geq 0. 
 \end{equation}
\end{lemma}

\medskip

The technique that we will implement in the next section require the knowledge about the $J$-image of the product of two polynomials $fg$. The polynomial  $J\left( fg\right)$ is then given by a Leibniz-type development \cite{PM-TAM-2019}  as mentioned in the next Lemma.

\begin{lemma}\label{lemmaJ(fg)} \cite{PM-TAM-2019} For any $f,g \in \mathcal{P}$, we have:
\begin{align}
\label{J(fg)}&J\left( f(x)g(x) \right)(x)= \sum_{n \geq 0 } J^{(n)}\left(f\right)(x)  \frac{g^{(n)}(x)}{n!} = \sum_{n \geq 0 } J^{(n)}\left(g\right)(x)  \frac{f^{(n)}(x)}{n!},
\end{align}
where the operator $J^{(m)}\, , \; m \geq 0$,  on $\mathcal{P}$ is defined by 
\begin{equation}\label{J^(m)}
J^{(m)}= \displaystyle \sum_{n \geq 0 } \frac{a_{n+m}(x)}{n!}D^{n}\,.
\end{equation}
\end{lemma}

\medskip

Let us suppose that $J$ is an operator expressed as in \eqref{operatorJ}, and acting as the derivative of order $k$, for some non-negative integer $k$, that is, it fulfils the following conditions.
\begin{align}
\label{deg-k1} & J\left( \xi^{k}\right)(x) = a_{0}^{[k]} \neq 0 \;\; \textrm{and} \; \;\deg\left(J\left( \xi^{n+k}\right)(x)  \right) = n,\quad n \geq 0;  \\
\label{deg-k2}& J\left( \xi^{i}\right)(x) =0,\quad 0 \leq i \leq k-1,\; \textrm{if} \; k\geq 1.
\end{align}

\begin{lemma}\label{J-descending}\cite{PM-TAM-2019} 
An operator $J$ fulfils \eqref{deg-k1}-\eqref{deg-k2} if and only if the next set of conditions hold.
\begin{description}
\item[a)] $a_{0}(x)=\cdots=a_{k-1}(x)=0$, if $k \geq 1$;
\item[b)] $\deg\left( a_{\nu}(x) \right) \leq \nu-k$, $\nu \geq k$;
\item[c)] \begin{equation}\label{neqcondition}
\lambda_{n+k}^{[k]} := \sum_{\nu = 0}^{n} \binom{n+k}{n+k-\nu} a_{n-\nu}^{[n+k-\nu]} \neq 0, \; n \geq 0.
\end{equation}
\end{description}
\end{lemma}

\begin{remark}
Note that in \eqref{neqcondition} we find $\lambda_{k}^{[k]}=a_{0}^{[k]} $.
\newline If $k=0$, then it is assumed that $\lambda_{n}^{[0]} \neq 0,\; n \geq 0$, matching \eqref{lambdan}, so that $J$ is an isomorphism. 
\newline If $k=1$, then $J$ imitates the usual derivative and is commonly called a lowering operator (e.g. \cite{Srivastava-Ben Cheikh,PM-TAM-2016}).
\end{remark}

Applying Lemma \ref{lemmaJ(fg)} to different pairs of polynomials, we obtain immediately the next identities.
\begin{align}
\label{Ezero}& J\left(xp(x)\right)= xJ\left(p(x)\right)+J^{(1)}\left( p(x) \right)\\
\label{Eone} &J\left(x^2p(x)\right)= x^2J\left(p(x)\right)+2xJ^{(1)}\left( p(x) \right)+J^{(2)}\left( p(x) \right) \\
\label{E3} &J\left(x^3p(x)\right)=x^3J\left(p(x)\right)+3x^2J^{(1)}\left( p(x) \right)+3x J^{(2)}\left( p(x) \right) + J^{(3)}\left( p(x) \right)
\end{align}

\begin{proposition}\label{AscOTxp}
Given an operator $J$ defined by  \eqref{operatorJ}, and taking into account the definition of the operator $J^{(m)}$, $ m \geq 0$:
 \begin{equation*}
J^{(m)}= \displaystyle \sum_{n \geq 0 } \frac{a_{n+m}(x)}{n!}D^{n}\,,
\end{equation*}
the following identities hold.
\begin{align}\label{J^{i}xp}
& J^{(i)}\left( x p(x) \right) = J^{(i+1)}\left( p(x) \right)+x J^{(i)}\left( p(x) \right)\, , \, \; i=0,1,2,\ldots.
\end{align}
\end{proposition}
\begin{proof}
Reading $i=0$ in \eqref{J^{i}xp} we find the identity stated in \eqref{Ezero}.
Let us now consider \eqref{Ezero} with $p(x)$ filled by the product $xp(x)$:
 $$J\left( x^2 p(x) \right) = J^{(1)}\left( xp(x) \right)+x J\left( xp(x) \right)\, .$$
The last term $xJ\left( xp(x) \right)$ can be rephrased taking into account  \eqref{Ezero}, yielding
\begin{align}\label{Edois}
J\left(x^2p(x)\right)= x^2J\left(p(x)\right)+xJ^{(1)}\left( p(x) \right)+J^{(1)}\left( xp(x) \right)\,.
\end{align}
Confronting \eqref{Eone} with \eqref{Edois}, we conclude  \eqref{J^{i}xp} with $i=1$:
$$ J^{(1)}\left( x p(x) \right) = J^{(2)}\left( p(x) \right)+x J^{(1)}\left( p(x) \right) . $$

Let us assume as induction hypotheses over $k \geq 2$ that
\begin{align*}
& J^{(i)}\left( x p(x) \right) = J^{(i+1)}\left( p(x) \right)+x J^{(i)}\left( p(x) \right)\, , \, \; i=0,\ldots, k-1.
\end{align*}
In view of Lemma \ref{lemmaJ(fg)}, we learn that for any polynomial $p=p(x)$
\begin{align*}
&J\left( x^{k+1} p\right) = \sum_{n \geq 0} J^{(n)}(p) \dfrac{ \left( x^{k+1} \right)^{(n)}}{n !} \; ;
\end{align*}
and thus we may write:
\begin{align}
\label{S1}&J\left( x^{k+1} p\right) = \sum_{\mu=0}^{k+1} J^{(\mu)}(p) \binom{k+1}{\mu} x^{k+1-\mu} \; ; \\
\label{S2} & J\left( x^{k} p\right) = \sum_{\nu=0}^{k}  J^{(\nu)}(p) \binom{k}{\nu} x^{k-\nu} \, .
\end{align}
Let us now consider \eqref{S2} with $p$ filled by the product $xp$ as follows:
\begin{align}
\label{S2-tilde} & J\left( x^{k+1} p\right) = \sum_{\nu=0}^{k}  J^{(\nu)}(xp) \binom{k}{\nu} x^{k-\nu} \, .
\end{align}
By means of the induction hypotheses, identity \eqref{S2-tilde} asserts the following.
\begin{align}
\nonumber &J\left( x^{k+1} p\right) =  \sum_{\nu=0}^{k-1}  \left( J^{(\nu+1)}(p)+ xJ^{(\nu)}(p)\right) \binom{k}{\nu} x^{k-\nu} + J^{(k)}(xp)\, \\
\nonumber &= \sum_{\nu=0}^{k-1}  J^{(\nu+1)}(p)\binom{k}{\nu} x^{k-\nu}  + \sum_{\nu=1}^{k-1} J^{(\nu)}(p) \binom{k}{\nu} x^{k+1-\nu} + J^{(k)}(xp) + J(p)x^{k+1}\, \\
\nonumber &= \sum_{\nu=0}^{k-2}  J^{(\nu+1)}(p)\left\{\binom{k}{\nu}+ \binom{k}{\nu+1} \right\}x^{k-\nu}  +  J^{(k)}(p) \binom{k}{k-1} x + J^{(k)}(xp) + J(p)x^{k+1}\, \\
\nonumber &= \sum_{\nu=0}^{k-2}  J^{(\nu+1)}(p)\binom{k+1}{\nu+1} x^{k-\nu}  +  J^{(k)}(p) k x + J^{(k)}(xp) + J(p)x^{k+1}\, \\
\nonumber &= \sum_{\nu=0}^{k-1}  J^{(\nu)}(p)\binom{k+1}{\nu} x^{k+1-\nu}  +  J^{(k)}(p) k x + J^{(k)}(xp)
\end{align}
In brief
\begin{align}
 \label{JwithinHyp}  J\left( x^{k+1} p\right) = \sum_{\nu=0}^{k-1}  J^{(\nu)}(p)\binom{k+1}{\nu} x^{k+1-\nu}  +  J^{(k)}(p) k x + J^{(k)}(xp) \, . 
\end{align}
Comparing \eqref{JwithinHyp} with \eqref{S1}, we get
\begin{align*}
& J^{(k)}(p)\binom{k+1}{k} x^{k+1-k} + J^{(k+1)}(p)\binom{k+1}{k+1} = kx J^{(k)}(p) + J^{(k)}(xp)\, \\
& \textrm{hence}\; \; x J^{(k)}(p) +  J^{(k+1)}(p) = J^{(k)}(xp)\, , 
\end{align*}
which ends the proof. 
\end{proof}

\section{An isomorphism applied to a $2$-orthogonal sequence }\label{sec:k=0 and J third order}

In the sequel, we consider that $J$ is an isomorphism and $a_{\nu}(x)=0\,,\; \nu \geq 4$, thus
\begin{align}
\label{third-order-J}& J = a_{0}(x)I + a_{1}(x)D+\frac{a_{2}(x)}{2}D^{2}+\frac{a_{3}(x)}{3!}D^{3}\, ,\, \textrm{where}\\
\nonumber & a_{0}(x)=a_{0}^{[0]}\; , \; a_{1}(x)=a_{0}^{[1]}+a_{1}^{[1]}x \; , \; a_{2}(x)=a_{0}^{[2]}+a_{1}^{[2]}x +a_{2}^{[2]}x^2 \; , \\
\nonumber & a_{3}(x)=a_{0}^{[3]}+a_{1}^{[3]}x +a_{2}^{[3]}x^2   + a_{3}^{[3]}x^3   \;, \; 
\end{align}
and we suppose that the MPS $\{P_{n} \}_{n \geq 0}$ is $2$-orthogonal and fulfils 
\begin{align}\label{J-image}
J\left( P_{n}(x) \right)= \lambda^{[0]}_{n}P_{n}(x), \; \textrm{with} \; \lambda^{[0]}_{n} \neq 0\, ,n \geq 0\, ,\end{align}
where
$$\lambda^{[0]}_{n} =a_{0}^{[0]}+\binom{n}{1}a_{1}^{[1]}+\binom{n}{2}a_{2}^{[2]}+ \binom{n}{3}a_{3}^{[3]}\, , \, n \geq 0 \,.$$
In view of $a_{\nu}(x)=0\,,\; \nu \geq 4\, ,$
 the operators
$J^{(1)}$, $J^{(2)}$ and $J^{(3)}$ have the following definitions as indicated in  \eqref{J^(m)}.
\begin{align}
\label{J^{(1)}} & J^{(1)}(p) = \left( a_{1}(x)I+a_{2}(x) D +\frac{1}{2!}a_{3}(x)  D^{2} \right) (p)  \\
\label{J^{(2)}}&  J^{(2)}(p)=\left( a_{2}(x)I+ a_{3}(x)D \right) (p)\\
\label{J^{(3)}} & J^{(3)}(p)=a_{3}(x)p\\
\nonumber & J^{(m)}(p) = 0\, , \; m \geq 4.
\end{align}

Broadly speaking, in this section we will intertwine the action of operators $J^{(k)}$, for initial values of $k$, with the simple multiplication by the monomial $x$, herein called $T_{x}$:
$$T_{x}: \, p \mapsto xp \; ,$$
in order to obtain the expansions of polynomials $J^{(1)}\left( P_{n}(x)\right)$, $J^{(2)}\left( P_{n}(x)\right)$ and $J^{(3)}\left( P_{n}(x)\right)$ in the basis formed by the $2$-orthogonal MPS $\{P_{n}(x)\}_{n \geq 0}$.

Most importantly, we review \eqref{RR2ortho}-\eqref{RR2ortho-ci} by establishing the following definition, considering henceforth $P_{-i}(x)=0$, $i=1,2,\ldots$.
\begin{align}\label{operatorTx}
T_{x}\left(P_{n}(x) \right) = P_{n+1}(x)+\beta_{n}P_{n}(x)+\alpha_{n}P_{n-1}(x)+\gamma_{n-1}P_{n-2}(x),\, n\geq 0\, .
\end{align}
Additionally, we can use the knowledge provided by Proposition \ref{AscOTxp}, valid for all operators not decreasing the degree \eqref{operatorJ}, that asserts
\begin{align}\label{J^{i}xp_Tx}
& J^{(i)}\left( T_{x}\left( p \right) \right) = J^{(i+1)}\left( p \right)+T_{x}\left( J^{(i)}\left( p \right)\right) \, , \, \; i=0,1,2,\ldots.
\end{align}
\medskip

\pagebreak
\noindent \texttt{\underline{First step}: applying $J$ to the four-term recurrence}
\medskip

Let us apply the operator $J$ to the recurrence relation \eqref{RR2ortho}, using both \eqref{J^{i}xp_Tx}, with $i=0$, and \eqref{J-image}:
\begin{align}
&\lambda^{[0]}_{n+2} T_{x}\left(P_{n+2}(x)\right)+ J^{(1)}\left( P_{n+2}(x)  \right) = \lambda^{[0]}_{n+3} P_{n+3}(x)\\
\nonumber &+\beta_{n+2}\lambda^{[0]}_{n+2} P_{n+2}(x)+\alpha_{n+2}\lambda^{[0]}_{n+1} P_{n+1}(x)+\gamma_{n+1}\lambda^{[0]}_{n}  P_{n}(x)\, .
\end{align}
Next, by \eqref{operatorTx}  we get $J^{(1)}\left( P_{n+2}(x) \right) $ in the basis $\{ P_{n}(x) \}_{ n \geq 0 }$:
\begin{align}\label{J1}
J^{(1)}\left( P_{n+2}(x)\right) &= \left(  \lambda^{[0]}_{n+3} - \lambda^{[0]}_{n+2}\right)   P_{n+3}(x) \\
\nonumber  + \alpha_{n+2}&\left(  \lambda^{[0]}_{n+1} - \lambda^{[0]}_{n+2} \right) P_{n+1}(x)+ \gamma_{n+1} \left( \lambda^{[0]}_{n}-\lambda^{[0]}_{n+2} \right) P_{n}(x)\, , n \geq 0.
\end{align}
Taking into account  the information retained in identities $J\left( P_{0}(x) \right)= \lambda^{[0]}_{0}P_{0}(x),$ 
$J\left( P_{1}(x) \right)= \lambda^{[0]}_{1}P_{1}(x)$, it is easy to verify that
\begin{align*}
&a_{1}(x)=J^{(1)}\left( P_{0}(x)\right) = \left(  \lambda^{[0]}_{1} - \lambda^{[0]}_{0}\right)   P_{1}(x) \, ,\\
&a_{1}(x)P_{1}(x)+a_{2}(x) = J^{(1)}\left( P_{1}(x)\right) = \left(  \lambda^{[0]}_{2} - \lambda^{[0]}_{1}\right)   P_{2}(x) +\alpha_{1}  \left( \lambda^{[0]}_{0}-\lambda^{[0]}_{1} \right) P_{0}(x)\, ,
\end{align*}
and thus we may define the image of every $P_{n}(x)$ through the operator $J^{(1)}$ as follows:
\begin{align}\label{J^(1)Operator}
J^{(1)}\left( P_{n}(x)\right) &= \left(  \lambda^{[0]}_{n+1} - \lambda^{[0]}_{n}\right)   P_{n+1}(x) \\
\nonumber + &\alpha_{n}\left(  \lambda^{[0]}_{n-1} - \lambda^{[0]}_{n} \right) P_{n-1}(x)+ \gamma_{n-1} \left( \lambda^{[0]}_{n-2}-\lambda^{[0]}_{n} \right) P_{n-2}(x)\; ,  n \geq 0.
\end{align}

\medskip
\noindent \texttt{\underline{Second step}: applying $J^{(1)}$ to the four-term recurrence}
\medskip

Let us now apply operator $J^{(1)}$ to the recurrence relation \eqref{RR2ortho} fulfilled by $\{P_{n}(x)\}_{n \geq 0}$:
\begin{align}\label{applyingJ(1)}
&J^{(1)} \left( T_{x}\left(P_{n+2}(x) \right) \right)= J^{(1)} \left(P_{n+3}(x)\right) \\
\nonumber &+\beta_{n+2}J^{(1)} \left(P_{n+2}(x)\right)+\alpha_{n+2}J^{(1)} \left(P_{n+1}(x)\right)+\gamma_{n+1}J^{(1)} \left(P_{n}(x)\right)\, .
\end{align}
We may then perform the following transformations:
\begin{align*}
G_{1}(n):\;\; & J^{(1)}\left( T_{x}\left( P_{n+2}(x) \right) \right) \rightarrow J^{(2)}\left( P_{n+2}(x) \right)+T_{x}\left( J^{(1)}\left( P_{n+2}(x) \right)\right) \, , \\
I_{1}(n): \;\;& J^{(1)}\left( P_{n}(x)\right)  \rightarrow \left(  \lambda^{[0]}_{n+1} - \lambda^{[0]}_{n}\right)   P_{n+1}(x) \\
\nonumber & \; + \alpha_{n}\left(  \lambda^{[0]}_{n-1} - \lambda^{[0]}_{n} \right) P_{n-1}(x)+ \gamma_{n-1} \left( \lambda^{[0]}_{n-2}-\lambda^{[0]}_{n} \right) P_{n-2}(x)\; , \\
M(n):\;\; & T_{x}\left(P_{n}(x) \right) \rightarrow P_{n+1}(x)+\beta_{n}P_{n}(x)+\alpha_{n}P_{n-1}(x)+\gamma_{n-1}P_{n-2}(x)\, .\,
\end{align*}
These transformations are defined in a suitable computer software, allowing a symbolic implementation that executes the adequate positive increments on the variable $n$. In this manner, it is possible for us to obtain the expansion of the image of $P_{n+2}(x)$ by operator $J^{(2)}$, in the basis $\{P_{n}(x)\}_{n \geq 0}$.

As a result of these computations,  \eqref{applyingJ(1)} corresponds to the next identity.
\begin{align}\label{J2}
&J^{(2)} \left(P_{n+2}(x) \right) = A_{n+4} P_{n+4}(x) +  B_{n+3} P_{n+3}(x) + C_{n+2} P_{n+2}(x)  \\
\nonumber &+D_{n+1} P_{n+1}(x) + F_{n} P_{n}(x)+ G_{n-1} P_{n-1}(x) + H_{n-2}P_{n-2}(x) \,  ,
\end{align}
where
\begin{align}
\label{second-step-coeffs}& A_{n} = \lambda^{[0]}_{n} -2 \lambda^{[0]}_{n-1}+ \lambda^{[0]}_{n-2}\, ;\\
\nonumber & B_{n} = \left(\beta_{n-1}-\beta_{n} \right)\left( \lambda^{[0]}_{n}-\lambda^{[0]}_{n-1}\right) \, ;\\
\nonumber & C_{n} = 2 \alpha_{n+1} \left( \lambda^{[0]}_{n}- \lambda^{[0]}_{n+1} \right)+  2 \alpha_{n} \left( \lambda^{[0]}_{n}- \lambda^{[0]}_{n-1} \right) \, ; \\
\nonumber & D_{n}=  \alpha_{n+1} \left(\beta_{n+1}-\beta_{n} \right) \left( \lambda^{[0]}_{n}- \lambda^{[0]}_{n+1} \right)\\
 \nonumber & \quad + \gamma_{n+1}\left(  \lambda^{[0]}_{n}-2 \lambda^{[0]}_{n+2}+ \lambda^{[0]}_{n+1} \right)+  \gamma_{n}\left(  \lambda^{[0]}_{n}-2 \lambda^{[0]}_{n-1}+ \lambda^{[0]}_{n+1} \right) \, ; \\
\nonumber  & F_{n} = \alpha_{n+2}\alpha_{n+1}\left( \lambda^{[0]}_{n} -2 \lambda^{[0]}_{n+1} + \lambda^{[0]}_{n+2} \right) + \gamma_{n+1} \left(\beta_{n+2}-\beta_{n} \right)\left( \lambda^{[0]}_{n} - \lambda^{[0]}_{n+2} \right) \, ; \\
 \nonumber & G_{n} = \alpha_{n+3}\gamma_{n+1} \left( \lambda^{[0]}_{n} -2 \lambda^{[0]}_{n+2} + \lambda^{[0]}_{n+3} \right) + \alpha_{n+1}\gamma_{n+2} \left( \lambda^{[0]}_{n} -2 \lambda^{[0]}_{n+1} + \lambda^{[0]}_{n+3} \right) \, ; \\
 \nonumber & H_{n} = \gamma_{n+3}\gamma_{n+1} \left( \lambda^{[0]}_{n} -2 \lambda^{[0]}_{n+2} + \lambda^{[0]}_{n+4} \right)\, . 
\end{align}
Once more, taking into account that $J\left( P_{i}(x) \right)= \lambda^{[0]}_{i}P_{i}(x),$ $i=0,1,$ and also  \eqref{J^(1)Operator} for $n=0,1,2$, we are able to confirm that the following initial identities hold: 
\begin{align*}
&J^{(2)} \left(P_{0}(x) \right) = A_{2} P_{2}(x) +  B_{1} P_{1}(x) + C_{0} P_{0}(x) \, , \\
&J^{(2)} \left(P_{1}(x) \right) = A_{3} P_{3}(x) +  B_{2} P_{2}(x) + C_{1} P_{1}(x) + D_{0} P_{0}(x)\,, 
\end{align*}
and, hence :
\begin{align}\label{J^(2)Operator}
&J^{(2)} \left(P_{n}(x) \right) = A_{n+2} P_{n+2}(x) +  B_{n+1} P_{n+1}(x) + C_{n} P_{n}(x)  \\
\nonumber &+D_{n-1} P_{n-1}(x) + F_{n-2} P_{n-2}(x)+ G_{n-3} P_{n-3}(x) + H_{n-4}P_{n-4}(x) \,  , \, n \geq 0\,. 
\end{align}

\medskip
\noindent \texttt{\underline{Third step}: applying $J^{(2)}$ to the four-term recurrence}
\medskip

Let us now apply operator $J^{(2)}$ to the recurrence relation \eqref{RR2ortho} fulfilled by $\{P_{n}(x)\}_{n \geq 0}$:
\begin{align*}
&J^{(2)} \left( T_{x}\left(P_{n+2}(x) \right) \right)= J^{(2)} \left(P_{n+3}(x)\right) \\
&+\beta_{n+2}J^{(2)} \left(P_{n+2}(x)\right)+\alpha_{n+2}J^{(2)} \left(P_{n+1}(x)\right)+\gamma_{n+1}J^{(2)} \left(P_{n}(x)\right)\, .
\end{align*}
We may perform the following transformations:
\begin{align*}
G_{2}(n):\;\; & J^{(2)}\left( T_{x}\left( P_{n+2}(x) \right) \right) \rightarrow J^{(3)}\left( P_{n+2}(x) \right)+T_{x}\left( J^{(2)}\left( P_{n+2}(x) \right)\right) \, , \\
I_{2}(n): \;\;& J^{(2)}\left( P_{n}(x)\right)  \rightarrow A_{n+2} P_{n+2}(x) +  B_{n+1} P_{n+1}(x) + C_{n} P_{n}(x)  \\
\nonumber &+D_{n-1} P_{n-1}(x) + F_{n-2} P_{n-2}(x)+ G_{n-3} P_{n-3}(x) + H_{n-4}P_{n-4}(x) \, , \\
M(n):\;\; & T_{x}\left(P_{n}(x) \right) \rightarrow P_{n+1}(x)+\beta_{n}P_{n}(x)+\alpha_{n}P_{n-1}(x)+\gamma_{n-1}P_{n-2}(x)\, .\,
\end{align*}
As before, these transformations and consequent simplifications, permit to express $J^{(3)}\left(P_{n+2}(x)\right)$ as follows.

\begin{align}\label{J(3)}
&J^{(3)} \left(P_{n+2}(x) \right) =  a_{3}^{[3]}P_{n+5}(x)\\
\nonumber &+\left( A_{n+4} \beta _{n+2}-A_{n+4} \beta _{n+4}-B_{n+3}+B_{n+4} \right)P_{n+4}(x) \\
\nonumber &+\left( A_{n+3} \alpha _{n+2}-A_{n+4} \alpha _{n+4}+B_{n+3} \beta _{n+2}-B_{n+3} \beta _{n+3}-C_{n+2}+C_{n+3} \right) P_{n+3}(x)\\
\nonumber &+ \left( A_{n+2} \gamma _{n+1}-A_{n+4} \gamma _{n+3}+B_{n+2} \alpha _{n+2}-B_{n+3} \alpha _{n+3}-D_{n+1}+D_{n+2}  \right)P_{n+2}(x)\\
\nonumber &+ \left( B_{n+1} \gamma _{n+1}-B_{n+3} \gamma _{n+2}+C_{n+1} \alpha _{n+2}-C_{n+2} \alpha _{n+2} \right.\\
\nonumber & \left. \hspace{2.5cm} -D_{n+1} \beta _{n+1}+D_{n+1} \beta _{n+2}-F_n+F_{n+1}  \right)P_{n+1}(x)\\
\nonumber &+\left(   C_n \gamma _{n+1}-C_{n+2} \gamma _{n+1}-D_{n+1} \alpha _{n+1}+D_n \alpha _{n+2} \right. \\
\nonumber &\left. \hspace{2.5cm} -F_n \beta _n+F_n \beta _{n+2}-G_{n-1}+G_n \right) P_{n}(x) \\
\nonumber &+\left( -D_{n+1} \gamma _n+D_{n-1} \gamma _{n+1}-F_n \alpha _n+F_{n-1} \alpha _{n+2} \right. \\
\nonumber &\left. \hspace{2.5cm} -G_{n-1} \beta _{n-1}+G_{n-1} \beta _{n+2}-H_{n-2}+H_{n-1}  \right)P_{n-1}(x) \\
\nonumber &+\left( -F_n \gamma _{n-1}+F_{n-2} \gamma _{n+1} \right. \\
\nonumber & \left. \hspace{2.5cm} -G_{n-1} \alpha _{n-1}+G_{n-2} \alpha _{n+2}-H_{n-2} \beta _{n-2}+H_{n-2} \beta _{n+2} \right)P_{n-2}(x)\\
\nonumber &+\left( -G_{n-1} \gamma _{n-2}+G_{n-3} \gamma _{n+1}-H_{n-2} \alpha_{n-2}+H_{n-3} \alpha _{n+2} \right) P_{n-3}(x)\\
\nonumber &+ \left( H_{n-4} \gamma _{n+1}-H_{n-2} \gamma _{n-3}\right) P_{n-4}(x)\;  , \; n \geq 0\, , 
\end{align}
with initial conditions:
\begin{footnotesize}
\begin{align*}
&J^{(3)} \left(P_{0}(x) \right) = a_{3}^{[3]}P_{3}(x)+\left( \left(\beta _0+\beta _1+\beta _2\right) a_{3}^{[3]}+a_{2}^{[3]}  \right)P_{2}(x)\\
&+\left(a_{3}^{[3]} \left(\alpha _1+\alpha _2+\beta _0^2+\beta _1 \beta _0+\beta _1^2\right)+\left(\beta _0+\beta _1\right) a_{2}^{[3]}+a_{1}^{[3]} \right)P_{1}(x)\\
&+\left(a^{[3]}_{3} \left(\alpha _1 \left(2 \beta _0+\beta _1\right)+\beta _0^3+\gamma _1\right)+\alpha _1 a^{[3]}_{2}+\beta _0 \left(\beta _0 a^{[3]}_{2}+a^{[3]}_{1}\right)+a^{[3]}_{0} \right)\, ;
\end{align*}
\begin{align*}
&J^{(3)} \left(P_{1}(x) \right) = a_{3}^{[3]}P_{4}(x)+\left( \left(\beta _1+\beta _2+\beta _3\right) a_{3}^{[3]}+a_{2}^{[3]}  \right)P_{3}(x)\\
&+\left( a^{[3]}_{3} \left(\alpha _1+\alpha _2+\alpha _3+\beta _1^2+\beta _2 \beta _1+\beta _2^2\right)+\left(\beta _1+\beta _2\right) a^{[3]}_{2}+a^{[3]}_{1}\right)P_{2}(x)\\
&+ \left( a^{[3]}_{3} \left(2 \left(\alpha _1+\alpha _2\right) \beta _1+\alpha _2 \beta _2+\beta _1^3+\gamma _1+\gamma _2\right)+\alpha _1 \beta _0 a^{[3]}_{3}+\left(\alpha _1+\alpha _2\right) a^{[3]}_{2} \right. \\
\nonumber &\hspace{2cm} \left. +\beta _1 \left(\beta _1 a^{[3]}_{2}+a^{[3]}_{1}\right)+a^{[3]}_{0}  \right)P_{1}(x)\\
\nonumber &+\left(\alpha _1 \left(a^{[3]}_{3} \left(\alpha _2+\beta _0^2+\beta _1 \beta _0+\beta _1^2\right)+\left(\beta _0+\beta _1\right) a^{[3]}_{2}+a^{[3]}_{1}\right)+\alpha _1^2 a^{[3]}_{3} \right.\\
\nonumber &\hspace{2cm} \left.  +\gamma _1 \left(\left(\beta _0+\beta _1+\beta _2\right) a^{[3]}_{3}+a^{[3]}_{2}\right) \right)\, .
\end{align*}
\end{footnotesize}

Recalling that $J^{(3)}\left(p\right)=a_{3}(x)p= \left( a_{3}^{[3]}x^3+a_{2}^{[3]}x^2 +a_{1}^{[3]}x+a_{0}^{[3]} \right) p$, identity \eqref{J(3)} enables the computation of the recurrence coefficients $\left(\beta_{n} \right)_{n \geq 0}$, $\left(\alpha_{n} \right)_{n \geq 1}$ and $\left(\gamma_{n} \right)_{n \geq 1} $ of a $2$-orthogonal $\{P_{n}\}_{n \geq 0}$ that is the solution of $J\left( P_{n} \right)= \lambda_{n}^{[0]}P_{n}(x) \, ,$ $ \, n \geq 0\, ,$ for a third-order $J$. We will pursuit with such computations in the next section for particular cases.

\section{Finding the 2-orthogonal solution of some third-order differential equations} \label{solutions}
Let us now assume that the $2$-orthogonal MPS $\{P_{n}\}_{ n \geq 0}$ fulfils $J\left( P_{n} \right)= \lambda_{n}^{[0]}P_{n}(x) \, ,$ $ \, n \geq 0$, where $J$ is defined by \eqref{operatorJ} with $a_{\nu}(x)=0\, ,\, \nu \geq 4$.
\medskip

\noindent Initially, we consider that $\deg\left(a_{3}(x) \right)=0$, though $a_{3}(x) \neq 0$, $\deg\left(a_{2}(x) \right) \leq 1$ and $\deg\left(a_{1}(x) \right)$=1. In other words:
\begin{align}
\label{third-order-J:a3=1}& \left( a_{0}(x)I + a_{1}(x)D+\frac{a_{2}(x)}{2}D^{2}+\frac{a^{[3]}_{0}}{3!}D^{3}\right) \left( P_{n}(x)\right) = \lambda^{[0]}_{n}P_{n}(x) \, ,\, \\
\nonumber & \textrm{with}\;\;   a_{0}(x)=a_{0}^{[0]}\; , \; a_{1}(x)=a_{0}^{[1]}+a_{1}^{[1]}x \; , \,a_{1}^{[1]} \neq 0, \; \\
\nonumber &a_{2}(x)=a_{0}^{[2]}+a_{1}^{[2]}x \; , \\
\nonumber & a_{3}(x)=a_{0}^{[3]} \neq 0  \;. \; 
\end{align}
Consequently, $\lambda_{n}^{[0]}= n a_{1}^{[1]} + a_{0}^{[0]}$, which we are assuming as nonzero for all non-negative integer $n$. Taking into account this set of hypotheses, identity \eqref{J(3)} provides several difference equations due to the linear independence of $\{P_{n}\}_{n \geq 0}.$
In particular, the coefficient of $P_{n+4}(x)$ on the right hand of \eqref{J(3)} is expressed by
$$-a^{[1]}_{1} \left(\beta _{n+2}-2 \beta _{n+3}+\beta _{n+4}\right)$$ and thus we get the equation
\begin{align}\label{Eq.1}
\beta _{n+4}-2\beta _{n+3}+\beta _{n+2}=0\,, \, n \geq 0\,. 
\end{align}
Also, the coefficients of $P_{n+3}(x)$ and $P_{n+2}(x)$ on  the right hand of \eqref{J(3)} provide the following two identities
\begin{align}
\label{Eq.2} a^{[1]}_{1}  \left(-2 \alpha _{n+2}+4 \alpha _{n+3}-2 \alpha _{n+4}+\left(\beta _{n+2}-\beta _{n+3}\right){}^2\right) &=0\, , \\
\label{Eq.3} -3 a^{[1]}_{1} \left(\gamma _{n+1}-2 \gamma _{n+2}+\gamma _{n+3}\right) &= a^{[3]}_{0}\,.
\end{align}

\medskip

\noindent Taking into account the results presented in the previous sections  \ref{operator} and \ref{sec:k=0 and J third order}, we have proved the following Proposition.

\medskip

\begin{proposition}\label{finalProp-case1}
Let us consider a $2$-orthogonal polynomial sequence $\{P_{n}\}_{n \geq 0}$ fulfilling
\begin{align*}
&J\left(P_{n}(x) \right) =  \lambda_{n}^{[0]}P_{n}(x) 
\end{align*}
where $J$ is defined by \eqref{operatorJ} with $a_{\nu}(x)=0\, ,\, \nu \geq 4,$
and such that  $a_{0}(x)=a_{0}^{[0]}\, , \, a_{1}(x)=a_{0}^{[1]}+a_{1}^{[1]}x \, , \,a_{1}^{[1]} \neq 0\,, \; $ $a_{2}(x)=a_{0}^{[2]}+a_{1}^{[2]}x \, , $ $\, a_{3}(x)=a_{0}^{[3]} \neq 0\, . $
\par Then the coefficient $a_{1}^{[2]}$ of polynomial $a_{2}(x)$ is zero and the recurrence coefficients of the sequence $\{P_{n}\}_{n \geq 0}$ are the following.
\begin{align}
\label{beta-case1}& \beta_{n}= -\frac{a_{0}^{[1]}}{a_{1}^{[1]}}\, , \, n \geq 0 \, ,\\
\label{alpha-case1}& \alpha_{n+1} = -\frac{a_{0}^{[2]}}{2a_{1}^{[1]}} (n+1)\, , \, n \geq 0 \, ,\\
\label{gamma-case1}& \gamma_{n+1} = -\frac{a_{0}^{[3]}}{a_{1}^{[1]}} \left(\frac{1}{3} + \frac{1}{2}n+\frac{1}{6}n^2 \right)= -\frac{a_{0}^{[3]}}{6 a_{1}^{[1]}} \left(n+1 \right)\left(n+2\right)\, , \, n \geq 0 \,.
\end{align}

Conversely, the $2$-orthogonal polynomial sequence $\{P_{n}\}_{n \geq 0}$ defined by the recurrence coefficients \eqref{beta-case1}-\eqref{gamma-case1} fulfils the third order differential equation 

\quad \quad $J\left(P_{n}(x) \right) =  \lambda_{n}^{[0]}P_{n}(x) \, ,$ $ \, n \geq 0\, ,$

\noindent where  $a_{0}(x)=a_{0}^{[0]}\, , \, a_{1}(x)=a_{0}^{[1]}+a_{1}^{[1]}x \, , \,a_{1}^{[1]} \neq 0\,, \; $ $a_{2}(x)=a_{0}^{[2]}\, , $ $\, a_{3}(x)=a_{0}^{[3]} \neq 0\, , $ and $a_{\nu}(x)=0\, ,\, \nu \geq 4$. 
\end{proposition}

\medskip

Concerning the assumptions of this last Propostion, it is worth mention that, later on, it is clarified in Proposition \ref{List-subcases} that if $\deg\left(a_{3}(x) \right)=0$, though $a_{3}(x) \neq 0$, and $\deg\left(a_{2}(x) \right) =0$, then $a_{1}^{[1]} \neq 0$.

\medskip

It is also important to remark that the $2$-orthogonal sequence described in Proposition \ref{finalProp-case1}  corresponds to a case, called E, of page 82 of \cite{Douak2classiques}. We then conclude that the single $2$-orthogonal polynomial sequence fulfilling the differential identity described in Proposition \ref{finalProp-case1} is classical in Hahn's sense, which means that the sequence of the derivatives $Q_{n}(x):=\frac{1}{n+1}DP_{n+1}(x)\, , \; n \geq 0 ,$ is also a $2$-orthogonal polynomial sequence. Furthermore, we read in \cite{Douak2classiques} (p. 104) that this sequence is an Appell sequence, in other words,  $Q_{n}(x)=P_{n}(x)\, , \; n \geq 0$. We review this detail while working with the intermediate relations \eqref{J1} and \eqref{J2} along the proof of Proposition \ref{finalProp-case1}, and based on those two identities we may indicate as corollary the following two differential identities.

\begin{corollary}\label{finalCorollary-case1}
Let us consider the $2$-orthogonal polynomial sequence $\{P_{n}\}_{n \geq 0}$ described in Proposition \ref{finalProp-case1} that fulfils
\begin{align*}
&J\left(P_{n}(x) \right) =  \lambda_{n}^{[0]}P_{n}(x) 
\end{align*}
where $J$ is defined by \eqref{operatorJ} with $a_{\nu}(x)=0\, ,\, \nu \geq 4,$ and such that  $a_{0}(x)=a_{0}^{[0]}\, , \, a_{1}(x)=a_{0}^{[1]}+a_{1}^{[1]}x \, , \,a_{1}^{[1]} \neq 0\,, \; $ $a_{2}(x)=a_{0}^{[2]}\, , $ $\, a_{3}(x)=a_{0}^{[3]} \neq 0\, . $
The sequence $\{P_{n}\}_{n \geq 0}$ also fulfils the following two identities
\begin{align*}
&\left( a_{1}(x)I+a_{0}^{[2]}D+\frac{1}{2}a_{0}^{[3]}D^2 \right) \left(  P_{n}(x) \right) = a_{1}^{[1]}P_{n+1}(x)\\
& \hspace{4cm} +\frac{1}{2}n\,a_{0}^{[2]}P_{n-1}(x)+\frac{1}{3} (n-1) n\,a_{0}^{[3]}P_{n-2}(x)\, ,\\
& \\
& DP_{n}(x)= nP_{n-1}(x)\; ,  n \geq 0\; , P_{-1}(x)=0\; .
\end{align*}
\end{corollary}
\medskip 

In the next proposition, we sum up a list of further conclusions pointed out by the application of the approach detailed in section \ref{sec:k=0 and J third order}. 
\begin{proposition}\label{List-subcases}
Let us consider a $2$-orthogonal polynomial sequence $\{P_{n}\}_{n \geq 0}$ fulfilling
\begin{align*}
&J\left(P_{n}(x) \right) =  \lambda_{n}^{[0]}P_{n}(x)\,, 
\end{align*}
where $J$ is defined by \eqref{operatorJ} with $a_{\nu}(x)=0\, ,\, \nu \geq 4 .$ 
\begin{itemize}
\item[a)] If $a_{2}(x)=a_{0}^{[2]} $ (constant) and $\deg\left( a_{3}(x) \right) \leq 2$, though $a_{3}(x) \neq 0$, then $a_{1}^{[1]} \neq 0 $.
\item[b)] If $a_{2}(x)= 0 $ and $\deg\left( a_{3}(x) \right) =1$, then there isn't a $2$-orthogonal polynomial sequence $\{P_{n}\}_{n \geq 0}$ such that $J\left(P_{n}(x) \right) =  \lambda_{n}^{[0]}P_{n}(x)\, , n \geq 0 $.
\item[c)] If $a_{3}(x)=0$, then the only solution of $J\left(P_{n}(x) \right) =  \lambda_{n}^{[0]}P_{n}(x) $ corresponds to $J=a_{0}^{[1]}D +  a_{0}^{[0]}I $.
\end{itemize}
\end{proposition}

Applying  the identities shown in sections \ref{operator} and \ref{sec:k=0 and J third order}, we are able to prove the forthcoming results. In particular we bring to light the $2$-orthogonal sequence defined in Corollary \ref{finalCor-case2}, as well as other differential identities fulfilled by a $2$-orthogonal solution besides the prefixed one $J\left(P_{n}(x) \right) =  \lambda_{n}^{[0]}P_{n}(x) \,,$ $\, n \geq 0\,.$ 

\medskip

In Theorem \ref{finalProp-case2}, we find the description of the $2$-orthogonal sequence that is the solution of the problem posed with respect to the third order operator $J$ defined by the conditions $a_{2}(x)=0$ and $\deg\left( a_{3}(x) \right) \leq 2$. Taking into consideration Proposition \ref{List-subcases}, we have assured that $\deg\left( a_{1}(x) \right)=1$, or $a_{1}^{[1]} \neq 0$.

\begin{theorem}\label{finalProp-case2}
Let us consider a $2$-orthogonal polynomial sequence $\{P_{n}\}_{n \geq 0}$ fulfilling
\begin{align*}
&J\left(P_{n}(x) \right) =  \lambda_{n}^{[0]}P_{n}(x) \, ,\, n \geq 0\,,
\end{align*}
where $J$ is defined by \eqref{operatorJ} with $a_{\nu}(x)=0\, ,\, \nu \geq 4,$ and such that $a_{0}(x)=a_{0}^{[0]}\, , \, a_{1}(x)=a_{0}^{[1]}+a_{1}^{[1]}x \, , \,a_{1}^{[1]} \neq 0\,, \; $ $a_{2}(x)=0\, , $ $\, a_{3}(x)=a_{0}^{[3]} +a_{1}^{[3]} x+a_{2}^{[3]} x^2\, . $
\par Then the recurrence coefficients of the sequence $\{P_{n}\}_{n \geq 0}$ are the following and the coefficients of the polynomial $a_{3}(x)= a_{2}^{[3]}x^2+a_{1}^{[3]}x+a_{0}^{[3]}$ fulfil $$(a_{1}^{[3]})^2-4a_{2}^{[3]}a_{0}^{[3]}=0.$$
\begin{align}
\label{beta-case2}& \beta_{n}= -\frac{a_{2}^{[3]}}{2 a_{1}^{[1]}}(n-1)n  - \frac{a_{0}^{[1]}}{a_{1}^{[1]}}\, , \, n \geq 0 \, ,\\
\label{alpha-case2}& \alpha_{n} = -\frac{a_{1}^{[3]}}{2 a_{1}^{[1]}} + \frac{a_{0}^{[1]}a_{2}^{[3]}}{(a_{1}^{[1]})^2}  +(n-2) \left( -\frac{3a_{1}^{[3]}}{4 a_{1}^{[1]}} + \frac{a_{2}^{[3]}\left( 9a_{0}^{[1]}+a_{2}^{[3]}\right) }{6(a_{1}^{[1]})^2} \right)\\
\nonumber & \hspace{0.6cm} + (n-2)^{2} \left(b_{0}+b_{1}(n-2)+b_{2}(n-2)^{2}\right)  \, , \, n \geq 1 \, ,\\
\label{gamma-case2}& \gamma_{n} =-\dfrac{1}{3a_{1}^{[1]}}\left(  a_{0}^{[3]} +\dfrac{a_{0}^{[1]}\left(  -a_{1}^{[1]}a_{1}^{[3]} + a_{0}^{[1]}a_{2}^{[3]}\right)}{(a_{1}^{[1]})^2}   \right) \\
\nonumber & \hspace{0.6cm} - (n-1) \left( \dfrac{(a_{1}^{[1]})^2 a_{0}^{[3]}-a_{0}^{[1]}a_{1}^{[1]}a_{1}^{[3]} + (a_{0}^{[1]})^2\, a_{2}^{[3]} }{2(a_{1}^{[1]})^3}\right)    \\
\nonumber & \hspace{0.6cm} + (n-1)^2 \left(  f_{0}+f_{1}(n-1) + f_{2}(n-1)^2 +f_{3}(n-1)^3 + f_{4}(n-1)^4 \right)\, , \, n \geq 1 \,;
\end{align}
where
\begin{align*}
&f_{0} = \frac{-18 a^{[3]}_{0} (a^{[1]}_{1})^2+6 a^{[3]}_{1}a^{[1]}_{1} \left(3 a^{[1]}_{0}+a^{[3]}_{2}\right) +a^{[3]}_{2} \left(-18 (a^{[1]}_{0})^2-12 a^{[3]}_{2} a^{[1]}_{0}+(a^{[3]}_{2})^2 \right)}{108 (a^{[1]}_{1})^3} \,,\\
& f_{1} = \frac{a^{[3]}_{2} \left(6 a^{[1]}_{1} a^{[3]}_{1}+a^{[3]}_{2} \left(a^{[3]}_{2}-12 a^{[1]}_{0}\right)\right)}{72 (a^{[1]}_{1})^{3}} \,,\\
& f_{2} = -\frac{a^{[3]}_{2} \left(a^{[3]}_{2} \left(12 a^{[1]}_{0}+a^{[3]}_{2}\right)-6 a^{[1]}_{1} a^{[3]}_{1}\right)}{216 (a^{[1]}_{1})^3 } \,,\\
& f_{3} = -\frac{(a^{[3]}_{2})^3}{72 (a_{1}^{[1]}) ^{3}} \,,\\
& f_{4} = -\frac{(a^{[3]}_{2})^3}{216 (a^{[1]}_{1})^3} \,,
\end{align*}
\begin{align*}
&b_{0}= \frac{1}{2} \left( - \frac{a_{1}^{[3]}}{2 a_{1}^{[1]}}  +  \frac{a_{0}^{[1]} a_{2}^{[3]}}{(a_{1}^{[1]})^2} + \frac{10\, (a_{2}^{[3]})^2}{12 \, (a_{1}^{[1]})^2}\right)\, , \\
&b_{1}= \frac{(a_{2}^{[3]})^2}{3 \, (a_{1}^{[1]})^2} \, , \\
&b_{2}=  \frac{(a_{2}^{[3]})^2}{12 \, (a_{1}^{[1]})^2} \, .
\end{align*}
Conversely, the $2$-orthogonal polynomial sequence $\{P_{n}\}_{n \geq 0}$ defined by the recurrence coefficients \eqref{beta-case2}-\eqref{gamma-case2}, under the assumption $\gamma_{n} \neq 0\, , \, n\geq1\,,$ fulfils the differential equation 
$J\left(P_{n}(x) \right) =  \lambda_{n}^{[0]}P_{n}(x) \, , \, n \geq 0\, ,$
where  $a_{0}(x)=a_{0}^{[0]}\, , \, a_{1}(x)=a_{0}^{[1]}+a_{1}^{[1]}x \, , \,a_{1}^{[1]} \neq 0\,, \; $ $a_{2}(x)=0\, , $ $a_{3}(x)= a_{2}^{[3]}x^2+a_{1}^{[3]}x+a_{0}^{[3]}$ with $(a_{1}^{[3]})^2-4a_{2}^{[3]}a_{0}^{[3]}=0,$
 and $a_{\nu}(x)=0\, ,\, \nu \geq 4$. 
\end{theorem}

\medskip

The content of Theorem \ref{finalProp-case2} provides an entire solution written in terms of the polynomial coefficients of the operator $J$. In the next Corollary we read a specific case endowed with Hahn's property, as we may prove analytically using the functionals of the dual sequence.

\begin{corollary}\label{finalCor-case2}
Let us consider the $2$-orthogonal polynomial sequence $\{P_{n}\}_{n \geq 0}$ fulfilling
\begin{align*}
&J\left(P_{n}(x) \right) =  \lambda_{n}^{[0]}P_{n}(x) \, ,\, n \geq 0\,,
\end{align*}
where $J$ is defined by \eqref{operatorJ} with $a_{\nu}(x)=0\, ,\, \nu \geq 4,$ and such that $a_{0}(x)=a_{0}^{[0]}\, , \, a_{1}(x)=\frac{1}{24}x \,, \; $ $a_{2}(x)=0\, , $ $\, a_{3}(x)=(x-1)^2\, . $
\par Then the recurrence coefficients of the sequence $\{P_{n}\}_{n \geq 0}$ are the following.
\begin{align}
\label{beta-case2.1}& \beta_{n}= -12(n-1)n \, , \, n \geq 0 \, ,\\
\label{alpha-case2.1}& \alpha_{n} = 12 (n-1) n (2 n-3)^2 \, , \, n \geq 1 \, ,\\
\label{gamma-case2.1}& \gamma_{n} = -4n(n+1)(2n-3)^2 (2n-1)^2 \, , \, n \geq 1.
\end{align}
Conversely, the $2$-orthogonal polynomial sequence $\{P_{n}\}_{n \geq 0}$ defined by the recurrence coefficients \eqref{beta-case2.1}-\eqref{gamma-case2.1} fulfils the differential equation 
$$\left( \frac{1}{6}(x-1)^2 D^3+\frac{1}{24}xD +a_{0}^{[0]} I\right) \left(P_{n}(x) \right) =  \lambda_{n}^{[0]}P_{n}(x) \, , \, n \geq 0\, ,$$
where  $\lambda_{n}^{[0]}=\frac{1}{24}n+  a_{0}^{[0]}  \, , \, n \geq 0\, $. 
\end{corollary} 
Furthermore, we remark that the polynomial sequence $\{P_{n}\}_{n \geq 0}$ defined by \eqref{beta-case2.1}-\eqref{gamma-case2.1}, fulfils the following two differential relations obtained by \eqref{J1} and \eqref{J2}.
\begin{align}
& \left( \frac{1}{24}xI+\frac{1}{2}(x-1)^2 D^2\right)(P_{n}(x)) = \frac{1}{24}P_{n+1}(x)\\
\nonumber &-\frac{1}{2}(3-2n)^{2}(n-1)nP_{n-1}(x)+ \frac{1}{3}(n-1)n(15-16n+4n^2)^2P_{n-2}(x)\, ,\\
\nonumber & \\
& (x-1)^2 D(P_{n}(x)) = nP_{n+1}(x) -2n(5+4n(2n-3))P_{n}(x)  \\
\nonumber & + (3-2 n)^2 n (24 (n-2) n+25) P_{n-1}(x) -8 (5-2 n)^2 (n-1) n (2 n-3)^3 P_{n-2}(x) \\
\nonumber & + 4 (3-2 n)^2 (5-2 n)^2 (7-2 n)^2 (n-2) (n-1) n P_{n-3}(x) \, , \, n \geq 0 \,,\, P_{-i}(x)=0\,.
\end{align}

\section*{Acknowledgements}

\noindent  This work was partially supported by
 CMUP (UIDB/00144/2020), which is funded by FCT (Portugal).


\end{document}